\documentclass[12pt]{amsart}
\usepackage{amsmath, amsthm, amscd, amsfonts}

\newtheorem{theorem}{Theorem}[section]

\newtheorem{proposition}[theorem]{Proposition}

\theoremstyle{definition}
\newtheorem{definition}[theorem]{Definition}
\newtheorem{example}[theorem]{Example}

\theoremstyle{remark}
\newtheorem{remark}[theorem]{Remark}
\numberwithin{equation}{section}

\begin{document}

\title{Fuzzy almost quadratic functions}
\author[A.K. Mirmostafaee, M.S. Moslehian]{A. K. Mirmostafaee$^1$ and  M. S. Moslehian$^2$}
\address{$^1$Alireza Kamel Mirmostafaee: Department of Mathematics, Ferdowsi
University, P. O. Box 1159, Mashhad 91775, Iran; \newline Banach
Mathematical Research Group (BMRG), Mashhad, Iran.}
\email{mirmostafaei@ferdowsi.um.ac.ir}
\address{$^2$Mohammad Sal Moslehian: Department of
Mathematics, Ferdowsi University, P. O. Box 1159, Mashhad 91775,
Iran;\newline Centre of Excellence in Analysis on Algebraic
Structures (CEAAS), Ferdowsi Univ., Iran.}
\email{moslehian@ferdowsi.um.ac.ir and moslehian@ams.org}



\begin{abstract}
We approximate a fuzzy almost quadratic function by a quadratic
function in a fuzzy sense. More precisely, we establish a fuzzy
Hyers--Ulam--Rassias stability of the quadratic functional equation
$f(x+y)+f(x-y)=2f(x)+2f(y)$. Our result can be regarded as a
generalization of the stability phenomenon in the framework of
normed spaces. We also prove a generalized version of fuzzy
stability of the Pexiderized quadratic functional equation
$f(x+y)+f(x-y)=2g(x)+2h(y)$.
\end{abstract} \maketitle

\section{introduction and preliminaries}
In order to construct a fuzzy structure on a linear space, A. K.
Katsaras \cite{KAT} defined the notion of fuzzy norm on  a linear
space. Later, a few  mathematicians have introduced and discussed
several notions of fuzzy norm from different points of view
\cite{FEL, K-S, X-Z}. In particular, T. Bag and S. K. Samanta
\cite{B-S1}, gave a new definition of a fuzzy norm in such a manner
that the corresponding fuzzy metric is of Kramosil and Michalek type
\cite{K-M}. They also studied some nice properties of the fuzzy norm
in \cite{B-S2}.

In mathematical analysis we may meet the following stability
problem: ``Assume that a function satisfies a functional equation
approximately according to some convention. Is it then possible to
find near this function a function satisfying the equation
accurately?'' In 1940, S. M. Ulam \cite{ULA} posed the first
stability problem. In the next year, D. H. Hyers \cite{HYE} gave a
partial affirmative answer to the question of Ulam. Hyers' theorem
was generalized by T. Aoki \cite{AOK} for additive mappings and by
Th. M. Rassias \cite{RAS1} for linear mappings by considering an
unbounded Cauchy difference. The paper \cite{RAS1} of Th. M. Rassias
has provided a lot of influence in the development of what we now
call Hyers--Ulam--Rassias stability of functional equations. We
refer the interested readers for more information on such problems
to e.g. \cite{B-M, CZE3, H-I-R, JUN, RAS3, RAS4}.

The functional equation
$$f(x+y)+f(x-y)=2g(x)+2h(y)$$
is said to be a Pexiderized quadratic functional equation. In the
case that $f=g=h$, it is called the quadratic functional equation.
The first author treating the stability of the quadratic equation
was F. Skof \cite{SKO} by proving that if $f$ is a mapping from a
normed space $X$ into a Banach space $Y$ satisfying
$\|f(x+y)+f(x-y)-2f(x)-2f(y)\|\leq \epsilon$ for some $\epsilon>0$,
then there is a unique quadratic function $g:X\to Y$ such that
$\|f(x)-g(x)\|\leq\frac{\epsilon}{2}$. P. W. Cholewa \cite{CHO}
extended Skof's theorem by replacing $X$ by an abelian group G.
Skof's result was later generalized by S. Czerwik \cite{CZE1} in the
spirit of Hyers--Ulam--Rassias. S.-M. Jung and P. K. Sahoo
\cite{J-S}, and K.W. Jun and Y. H. Lee \cite{J-L} proved the
stability of quadratic equation of Pexider type. The stability
problem of the quadratic equation has been extensively investigated
by a number of mathematicians, see \cite{CZE2, CZE3, J-S, M-M, RAS2}
and references therein.

In this paper, we use the definition of a fuzzy normed space given
in \cite{B-S1} to exhibit two reasonable fuzzy versions of
stability for (Pexiderized) quadratic functional equation in the
fuzzy normed linear space setting. More precisely, we approximate
a function $f$ form a space $X$ to a fuzzy Banach space $Y$ by a
quadratic function $Q:X \to Y$ in a fuzzy sense. In fact, we
obtain a fuzzy Hyers--Ulam--Rassias stability of the quadratic
equation in section 2 and a generalized version of fuzzy
stability of a Pexiderized quadratic equation in section 3,.

Some fuzzy stability results have been already established for the
Cauchy equation $f(x+y)=f(x)+f(y)$ in \cite{KM-M} and for the Jensen
equation in \cite{M-M-M}.

Following \cite{B-S1}, we give our notion of a fuzzy norm.
%--------------------------------------------------------------------------------------------------%
\begin{definition} Let $X$ be a real linear space. A function $N \colon   X \times {\mathbb R} \to [0,1]$ (the so-called
fuzzy subset) is said to be a fuzzy norm on $X$ if for all $x, y
\in X$ and all $s, t \in {\mathbb R}$,

(N1) $N(x, c) = 0$ for $c \leq 0$;

(N2) $x=0$ if and only if $N(x,c) = 1$ for all $c>0$;

(N3) $N(cx,t) = N(x, \frac{t}{|c|})$ if $c \neq 0$;

(N4) $N(x + y, s + t) \geq \min \{N(x, s), N(y, t)\}$;

(N5) $N(x, .)$ is a non-decreasing function on ${\mathbb R}$  and
$\lim_{t \to \infty}N(x, t) = 1$.

The pair $(X,N)$ is called a fuzzy normed linear space. One may
regard $N(x, t)$ as the truth value of the statement `the norm of
x is less than or equal to the real number t'.
\end{definition}
%--------------------------------------------------------------------------------------------------%
\begin{example}\label{exam}
Let $(X, \|.\|)$ be a normed linear space. One can be easily
verify  that for each $k > 0$,
$$N_k (x, t) =\left \{ \begin{array}{ll}
\frac{t}{t+ k \|x\|} & \qquad t> 0   \\
0 &\qquad t\leq 0
\end{array}\right.$$
defines a fuzzy norm on $X$.
\end{example}
%--------------------------------------------------------------------------------------------------%
\begin{example} \label{examm} Let $(X, \|.\|)$ be a normed linear space. Then
$$N(x, t) =\left \{ \begin{array}{ll}
0 &\qquad t\leq \|x\| \\
1 &\qquad  t > \|x\|
\end{array}\right.$$
is a fuzzy norm on $X$.
\end{example}
%----------------------------------------------------------------------------------
Let $(X,N)$ be a fuzzy normed linear space. Let $\{x_n\}$ be a
sequence in $X$. Then $\{x_n\}$ is said to be convergent if there
exists $x\in X$ such that $$lim_{n\to \infty} N(x_n-x, t) = 1$$
for all $t > 0$. In that case, $x$ is called the limit of the
sequence $\{x_n\}$ and we denote it by $N-lim x_n=x$.
%--------------------------------------------------------------------------------------------------%

A sequence $\{x_n\}$ in $X$ is called Cauchy if for each
$\varepsilon>0$ and each $t >0$ there exists $n_0$ such that for
all $n \geq n_0$ and all $p>0$ we have $N(x_{n+p}-x_n,
t)>1-\varepsilon$.

It is known that every convergent sequence in a fuzzy normed space
is Cauchy. If each Cauchy sequence is convergent, then the fuzzy
norm is said to be complete and the fuzzy normed space is called a
fuzzy Banach space.
%--------------------------------------------------------------------------------------------------%

\section{Fuzzy Hyers--Ulam--Rassias stability of the quadratic equation}

Let $f$ be a function from a fuzzy normed space $(X, N)$ into a
fuzzy Banach space $(Y, N')$ and $q \neq \frac{1}{2}$. The
function $f$ is called a \emph{fuzzy $q$-almost quadratic
function}, if
\begin{equation}\label{pq}
N'(f(x+y) + f(x-y) - 2f(x) - 2f(y), t+s) \geq \min \{N(x, t^q),
N(y, s^q )\}
\end{equation}
for all $x, y \in X$ and all $s, t \in [0, \infty)$.

The following result gives a Hyers--Ulam--Rassias stability of the
quadratic equation $f(x+y) + f(x-y) = 2f(x) + 2f(y)$.

\begin{theorem} \label{t1}
Let $q > {1 \over 2}$ and $f$ be a fuzzy $q$-almost quadratic
function from a fuzzy normed space $(X, N)$ into a fuzzy Banach
space $(Y, N')$. Then there is a unique quadratic  function $Q:X
\to Y$ such that for each $x \in X$,
\begin{equation}\label{Tfts}
N'(Q(x) - f(x), t) \geq  N(x, (\frac{2^{2-p}-1}{4})^q t^q) \qquad
(x \in X, t >0),
\end{equation}
where $p= {1\over q}$.
\end{theorem}
\begin{proof} Put $x=y$ and $s=t$ in (\ref{pq}) to obtain
\begin{equation}\label{mn}
N'( f( 2x ) - 4 f( x ),  2 t ) \geq N( x, t^q)  \qquad (x \in X,
t>0).
\end{equation}
Replacing $x$ by $2^n x$ in (\ref{mn}), we see that
\begin{equation}
N'(f(2^{n+1} x) - 4 f(2^n x), 2t)  \geq N(x, {t^q \over 2^n})
\qquad (x \in X, n \geq 0, t>0).
\end{equation}
It follows that
\begin{equation*}
N'(f(2^{n+1} x) - 4 f(2^n x), 2^{{n\over q}+1}t^{ 1\over q})  \geq
N(x, t) \qquad (x \in X, n \geq 0, t>0).
\end{equation*}
Whence
\begin{equation}
N'({ f( 2^{n+1} x ) \over 4^{n+1}} - {f ( 2^n x ) \over 4^n}, {
t^p 2^{n(p -2)+1}} ) \geq N ( x, t) \qquad (x \in X, n \geq 0,
t>0)
\end{equation}
where $p = {1 \over q}$. If $n > m \geq 0$, then
\begin{eqnarray}\label{Cau}
N'({f( 2^n x) \over 4^n} -  {f( 2^m x) \over 4^m}&,&
\sum^n_{k=m+1} t^p 2^{n(p -2)+1} )\nonumber\\ &\geq& N'(
\sum^n_{k=m+1}({f( 2^k x) \over 4^k}- {f( 2^{k-1}x) \over
4^{k-1}}), \sum^n_{k=m+1} t^p 2^{k(p -2)+1}) \nonumber\\ &\geq&
\min \bigcup_{k=m+1}^n \{ N'({f( 2^k x) \over
4^k} - {f( 2^{k-1} x) \over 4^{k-1}}, t^p 2^{k(p -2)+1})\\
&\geq& N(x, t) \qquad (x \in X, t>0)\nonumber.
\end{eqnarray}
Let $c >0$ and $\varepsilon$ be given. Since $\lim_{t \to \infty}
N(x, t) = 1$, there is some $t_0 > 0$ such that
$$N(x, t_0)\geq 1 - \varepsilon.$$
Fix some $t > t_0$. The convergence of the series
$\sum_{n=1}^\infty t^p 2^{n(p -2)+1}$ guarantees that  there
exists some $n_0 \geq 0$ such that for each $n > m \geq n_0$, the
inequality  $\sum_{k = m+1} ^n t^p 2^{k(p -2)+1} < c$ holds. It
follows that,
\begin{eqnarray*}
N'({f( 2^n x) \over 4^n} -  {f( 2^m x) \over 4^m}, c)
&\geq& N'({f( 2^n x) \over 4^n} -  {f( 2^m x) \over 4^m},\sum_{k = m+1}^n  t_0^p 2^{k(p -2)+1}) \\
&\geq&  N(x, t_0) \\ &\geq& 1 - \varepsilon .
\end{eqnarray*}
Hence $\{ { f(2^n x ) \over  4^n}\}$ is a Cauchy sequence in $(Y,
N') $. Since $(Y, N')$ is a fuzzy Banach space, this sequence
converges to some $Q(x) \in Y$. Hence, we can define a mapping
$Q:X \to Y$, by $Q(x) := N'-\lim_{n \to \infty} { f(2^n x ) \over
4^n }$. Moreover, if we put $m=0$ in (\ref{Cau}) we observe that
$$
N'({f( 2^n x) \over 4^n} - f(x), \sum_{k =1}^n  t^p 2^{k(p -2)+1})
\geq N( x, t).
$$
Therefore,
\begin{equation}\label{app}
N'({f( 2^n x) \over 4^n} - f(x), t ) \geq N( x, {t^q \over
(\sum_{k =1}^n   2^{k(p -2)+1})^q} ).
\end{equation}
Next we will show that $Q$ is quadratic. Let $x,y \in X$, then we
have
\begin{eqnarray*}
&&N'(Q(x+y) + Q(x-y) - 2 Q(x) - 2 Q(y), t )\geq \\
&\min& \{ N'(Q(x+y) - \frac{f(2^n(x+y))}{4^n}, \frac{t}{5}),
N'(T(x-y) - \frac{f(2^n(x-y))}{4^n}, \frac{t}{5}),\\&& N'( 2{f(2^n
x)\over 4^n}- 2 T(x), {t\over 5}), N'(2{f(2^ny)\over 4^n}- 2T(y),
{t\over 5}),\\&& N'({f(2^n(x+y))\over 4^n} - {f(2^n(x-y))\over
4^n} - 2{f(2^n(x))\over 4^n} - 2 {f(2^n(y))\over 4^n}, {t\over
5})\}.
\end{eqnarray*}
The first four terms on the right hand side of the above
inequality tend to $1$ as $ n \to \infty$ and the fifth term, by
(\ref{pq}) is greater than or equal to
$$\min \{ N(2^nx, (\frac{4^nt}{10})^q), N(2^ny, (\frac{4^nt}{10})^q\}=\min \{ N(x, 2^{(2q-1)n}(\frac{t}{10})^q),
N(y, 2^{(2q-1)n}(\frac{t}{10})^q)\},$$ which tends to $1$ as $n
\to \infty$. Therefore
$$N'(Q(x+y) + Q(x-y) -2 Q(x) - 2 Q(y),   t)=1$$
for each $x,y \in X $ and $t > 0$. This means that $Q(x+y) +
Q(x-y) = 2 Q(x) + 2 Q(y)$ for each $x, y \in X$. Next we
approximate the difference between $f$ and $Q$ in a fuzzy sense.
For every $x \in X$ and  $t, s > 0$, by (\ref{app}), for large
enough $n$, we have
\begin{eqnarray*}
N'(Q(x) - f(x), t) &\geq& \min \{ N'(Q(x) - { f(2^n x) \over 4^n},
\frac{t}{2}), N'({f(2^n x) \over 4^n} - f(x), \frac{t}{2})\}\\
&\geq& N( x, {t^q \over (\sum_{k =1}^n 2^{k(p -2)+2})^q})\\
&\geq& N(x, (\frac{2^{2-p}-1}{4})^q t^q).
\end{eqnarray*}
Let $Q'$ be another quadratic function from $X$ to $Y$ which
satisfies (\ref{Tfts}). Since for each $n \in \mathbb{N}$,
$$Q(2^n x) = 4^n Q(x) ~\& ~Q'(2^n x) = 4^n Q'(x),$$
we have
\begin{eqnarray*}
N'(Q(x) - Q'(x), t ) &=& N'(Q(2^n x) - Q'(2^nx), 4^n t)\\
&\geq& \min \{ N'(Q'(2^n x) - f(2^nx), \frac{4^n t}{2}),
N'(f(2^nx)- Q(2^n x)),\frac{4^n t}{2})\\ &\geq& N(2^n
x,(\frac{2^{2-p}-1}{4})^q 4^{(n- \frac{1}{2})q}t^q)\\ &=& N(x,
(\frac{2^{2-p}-1}{4})^q \frac{4^{nq}t^q}{2^n 2^q})
\end{eqnarray*} for each $n \in
\mathbb{N}$. Due to $q > \frac 1 2$,  $\lim_{n \to \infty} N(x,
(\frac{2^{2-p}-1}{4})^q \frac{4^{nq}t^q}{2^n 2^q})=1$ for each $x
\in X$ and $t > 0$. Therefore $Q = Q'$.
\end{proof}
%----------------------------------------------------------------------------------------------%

\begin{remark}
If $N'(Q(x) - f(x), ~.~)$ is assumed to be right continuous at
each point of $(0, \infty)$ then we get a better fuzzy
approximation than (\ref{Tfts}) as follows.

We have
\begin{eqnarray*}
N'(Q(x) - f(x), t+s) &\geq& \min \{ N'(Q(x) - { f(2^n x) \over
4^n}, s), N'({f(2^n x) \over 4^n} - f(x), t)\}\\ &\geq& N( x,
{t^q \over (\sum_{k =1}^n 2^{k(p -2)+1})^q})\\ &\geq& N(x,
(\frac{2^{2-p}-1}{2})^q t^q).
\end{eqnarray*}
Tending $s$ to zero we infer
$$N'(Q(x) - f(x), t) \geq N(x, (\frac{2^{2-p}-1}{2})^q t^q)\qquad (x \in X, t >0).$$
\end{remark}
%----------------------------------------------------------------------------------------------%
\begin{example} Let $X$ be a normed algebra. Using the notation of Example \ref{exam}, let $N= N_1$ and
$N^{'}= N_2$. Define $f:(X, N) \to (X, N^{'}) $ by $f(x) = x^2
+||x|| x_0$, where $x_0$ is a unit vector in $X$. A
straightforward computation shows that
$$f(x+y) + f(x-y) - 2 f(x) - 2 f(y) = (|| x+y || + || x - y || - 2
|| x || - 2 || y ||) x_0 $$ and
\begin{eqnarray*} N(f(x+y) +
f(x-y) - 2 f(x) - 2 f(y), s+t) \geq \min \{ N^{'} (x,t), N^{'}(y,
s) \}.
\end{eqnarray*}
Therefore the conditions of Theorem \ref{t1} for $q = 1$ holds.
The fuzzy difference between $Q(x) = \lim_{n \to \infty}\frac{
f(2^n x)}{ 4^n} = x^2 $ and $f(x)$ is equal to $$N(f(x) - Q(x), t
) = {t \over { t+ || x||}} = N(x,t) \geq N^{'}(x, t) \geq N^{'}(x,
{t\over 2}).
$$
\end{example}
%----------------------------------------------------------------------------------------------%

Using Example \ref{examm}, Theorem \ref{t1} can be regarded as a
generalization of the classical stability result in the framework
of normed spaces (see \cite{H-I-R}).

\begin{theorem}
Let $f$ be a function from a  normed space $(X, \|.\|))$ into a
Banach space $(Y, |||.|||)$. Let for some $p > { 2}$,
\begin{equation}
|||(f(x+y) + f(x-y) - 2f(x) - 2f(y)||| \leq \|x\|^p + \|y\|^p
\end{equation}
for all $x, y \in X$. Then there is a unique quadratic additive
function $Q: X \to Y$ such that
\begin{equation}
|||T(x) - f(x)||| \leq  \frac{4}{2^{2-p}-1}||x||^p \qquad (x \in
X).
\end{equation}
\end{theorem}

%----------------------------------------------------------------------------------------------%

\begin{remark}
Using the Hyers' type sequence $\{4^nf(2^{-n}x)\}$ one can get
`dual' versions of Theorem \ref{t1} when $q < \frac{1}{2}$.
\end{remark}
%----------------------------------------------------------------------------------------------%
\section{A fuzzy general stability of the Pexiderized quadratic equation}

In this section, we generalize the norm version of stability of a
Pexiderized quadratic equation to the framework of fuzzy normed
spaces. Due to some technical reasons, we first examine the
stability for odd and even functions and then we apply our results
to a general function.

Throughout this section we assume that $X$ is a linear space, $(Y,
N)$ is a fuzzy Banach space and $( Z, N^{'})$ is a fuzzy normed
space. In addition, we suppose that $\varphi : X \times X \to Z$
is a mapping such that \begin{eqnarray}\label{p2al} \varphi(2x,
2y) = \alpha \varphi(x, y) \end{eqnarray} for some $\alpha \in
{\mathbb R}$ and all $x, y \in X$.

\begin{proposition}\label{p1}
Suppose that $0< |\alpha |< 2$ and that $f, g$ and $h$ are odd
functions from $X$ to $Y$ such that
\begin{eqnarray} \label{001}
N( f(x+y) + f(x-y) -2 g(x) - 2h(y), t) \geq N'(\varphi(x,y), t)
\end{eqnarray}
for each $x, y \in X$ and $t \in \mathbb{R}$. Then there is a
unique additive mapping $T: X \to X$ such that
\begin{eqnarray*}
N(f(x)-T(x), t) \geq N''(x, \frac{2-|\alpha|}{4}t),
\end{eqnarray*}
\begin{eqnarray}\label{ghT}
N(g(x)+h(x)-T(x), t) \geq N''(x,
\frac{6-3|\alpha|}{14-|\alpha|}t),
\end{eqnarray}
where
\begin{eqnarray}\label{106}
N^{''}(x, t) =  \min \{ N'(\varphi(x,x), t/3), N'(\varphi(x,0),
t/3), N'(\varphi(0,x), t/3)\}.
\end{eqnarray}
\end{proposition}
%----------------------------------------------------------------------------------------------%
\begin{proof} Noting to (N3), it is sufficient to prove the theorem in the case that $0 < \alpha <2$.
By changing the roles of $x$ and $y$ in (\ref{001}),  we get
\begin{eqnarray}\label{002}
N( f(x+y) - f(x-y) -2 g(y) - 2h(x), t) \geq N'(\varphi(y,x), t).
\end{eqnarray}
It follows from (\ref{001}), (\ref{002}) and (N4) that
\begin{eqnarray}\label{003}
N( f(x+y)  - g(x) - h(y) - g(y) - h(x), t) &\geq&\nonumber \\
&\min& \{ N'(\varphi(x,y), t), N'(\varphi(y,x), t)\}.
\end{eqnarray}
If we put $y = 0$ in (\ref{003}), we obtain
\begin{eqnarray}\label{004}
N( f(x)  - g(x) -  h(x), t) \geq \min \{ N'(\varphi(x,0), t),
N'(\varphi(0,x), t)\}.
\end{eqnarray}
It follows from ( \ref{003}), ( \ref{004}) and (N4)
\begin{eqnarray} \label{005}
N( f(x+y)  - f(x) - f(y), 3t) \geq  \min &&\{ N'(\varphi(x,y), t),
N'(\varphi(y,x), t),\nonumber \\ && N'(\varphi(x,0), t),
N'(\varphi(0,x), t), \\ && N'(\varphi(y, 0), t), N'(\varphi(0, y),
t)\}\nonumber.
\end{eqnarray} Then
$N^{''}(2^n x, t ) = N^{''}( x, \frac{t}{\alpha^n})$. If we put
$x=y$ in ( \ref{005}), we see that
\begin{eqnarray}\label{006}
N( f(2x) - 2f(x), t) \geq N^{''}(x,t).
\end{eqnarray}
Replacing $x$ by $2^n x$ in (\ref{006}) we have
\begin{eqnarray*}
N( {f(2^{n+1}x)\over 2^{n+1}} - {f(2^n x)\over 2^n}, t) &=& N(
f(2^{n+1}x) - f(2^n x), 2^nt)\\& \geq& N^{''}(2^n x, 2^n t)\\
&\geq&N^{''}(x, (\frac{2}{\alpha})^nt),
\end{eqnarray*}
whence
\begin{eqnarray*}
N({f(2^{n+1}x)\over 2^{n+1}} - {f(2^n x)\over
2^n},(\frac{\alpha}{2})^nt) \geq N{''}(x, t).
\end{eqnarray*}
Therefore for each $n > m \geq 0$,
\begin{eqnarray}\label{009}
N( {f(2^{n}x)\over 2^{n}} - {f(2^m x)\over 2^m}, \sum_{k=m+1}^n
(\frac{\alpha}{2})^{k-1}t) &=& N ( \sum_{k=m+1}^n {f(2^{k}x)\over
2^{k}} - {f(2^{k-1} x)\over 2^{k-1}}, \sum_{k=m+1}^n
(\frac{\alpha}{2})^{k-1}t) \nonumber \\ &\geq& \min
\bigcup_{k=m+1}^n \{ N({f(2^{k}x)\over 2^{k}} - {f(2^{k-1} x)\over
2^{k-1}}, (\frac{\alpha}{2})^{k-1}t) \}\\ &\geq& N{''}(x,
t).\nonumber
\end{eqnarray}
Let $t_0 > 0$ and $\varepsilon >0$ be given. Thanks to the fact
that $\lim_{s \to \infty} N^{''}(x, s) = 1$, we can find some
$t_1 > t_0$ such that
\begin{eqnarray*}
N^{''}(x, t_1) > 1 - \varepsilon. \end{eqnarray*} By the
convergence of the series $\sum^\infty_{n=1}(\frac{\alpha}{2})^n
t_1$ we can find  some $n_0 \in \mathbb{N}$ such that for each $n
> m \geq n_0$,
\begin{eqnarray*}
\sum_{k=m+1}^n (\frac{\alpha}{2})^{k-1}t_1 < t_0.
\end{eqnarray*}
Therefore
\begin{eqnarray*}
N( {f(2^{n}x)\over 2^{n}} - {f(2^m x)\over 2^m}, t_0 ) &\geq& N(
{f(2^{n}x)\over 2^{n}} - {f(2^m x)\over 2^m},
\sum_{k=m+1}^n(\frac{\alpha}{2})^{k-1}t_1)
\\ &\geq& N^{''}(x, t_1) \\
&>& 1 - \varepsilon.
\end{eqnarray*}
Hence, $\{ { f(2^n x ) \over 2^n} \}$ is a Cauchy sequence in $(Y,
N)$. Since $(Y, N)$ is a Banach fuzzy space, this sequence
converges to some point $T(x) \in Y$. Define $T: X \to Y$ by
$T(x):= N- \lim_{n\to \infty} { f( 2^n x) \over 2^n}.$

Fix $x, y \in X$ and $t>0$. It follows from (\ref{005}) that
\begin{eqnarray}\label{013}
N({f(2^n(x+y))\over 2^n} - {f(2^n x)\over 2^n} - {f(2^n y)\over
2^n}, \frac{t}{4})  &=& N ( f(2^n(x+y)) - f(2^n x) - f(2^n y), \frac{2^n t}{4} ) \nonumber \\
&\geq& \min \{ N'( \varphi(x, y ), {2^n t\over 12 \alpha^n}), N'(
\varphi( y, x ), {2^n t\over 12 \alpha^n}), \\&& N'( \varphi(x, 0
), {2^n t\over 12 \alpha^n}),  N'( \varphi(0, x ), {2^n t\over 12
\alpha^n}), \nonumber \\ && N'( \varphi(y, 0 ), {2^n t\over 12
\alpha^n }),  N'( \varphi(0, y ), {2^n t\over 12
\alpha^n})\}\nonumber
\end{eqnarray}
for all $n$. Moreover,
\begin{eqnarray}\label{014}
N(T(x+y) - T(x) - T(y), t ) &\geq& \min\{ N( T(x+y) - {
f(2^n(x+y)) \over 2^n}, {t\over 4} ),\nonumber \\&& N(T(x) -
{f(2^n x)\over 2^n}, {t\over 4}), N(T(y) - {f(2^n y)\over 2^n},
{t\over 4}), \\&& N({f(2^n(x+y))\over 2^n} - {f(2^n x)\over 2^n}
- {f(2^n y)\over 2^n}, {t\over 4})\} \nonumber
\end{eqnarray}
for all $n$. Since each factor in the right hand side of
(\ref{013}) and (\ref{014}) tends to $1$ as $n \to \infty$, one
can easily see that $N( T(x+y) - T(x) - T(y), t )=1$ whence
$$T(x+y)= T(x) + T(y).$$
Furthermore, using (\ref{009}) with $m = 0$, we see that for large
$n$,
\begin{eqnarray}\label{015}
N( T(x) - f(x), t) &\geq& \min \{ N(T(x) - {f(2^n x)\over 2^n},
{t\over 2}), N( {f(2^n x)\over 2^n} - f(x), {t \over 2}) \}\nonumber \\
&\geq& \min \{ N(T(x) - {f(2^n x)\over 2^n}, {t\over 2}),
N^{''}(x, {t \over 2 \sum^n_{k=1}({\alpha \over 2})^{k-1}})\}\\
&\geq& N^{''}(x, {t \over 2 \sum^\infty_{k=0}( {\alpha \over
2})^{k-1}})\nonumber\\
&=& N^{''}(x, \frac{2-\alpha}{4}t).\nonumber
\end{eqnarray}
It follows from (\ref{004}) and (\ref{015}) that
\begin{eqnarray*}
N(g(x)+h(x)-T(x)&,& \frac{14-\alpha}{12}t)\\&\geq& \min\{N(f(x) -
T(x),t), N(g(x)+h(x)-f(x), \frac{2-\alpha}{12}t) \}\\&\geq&
\min\{N^{''}(x, \frac{2-\alpha}{4}t),  N'(\varphi(x,0),
\frac{2-\alpha}{12}t), N'(\varphi(0,x), \frac{2-\alpha}{12}t)\}\\
&\geq& N^{''}(x, \frac{2-\alpha}{4}t),
\end{eqnarray*}
whence we obtained (\ref{ghT}).

The proof for the  uniqueness assertion is similar to  Theorem
\ref{t1}.
\end{proof}
%----------------------------------------------------------------------------------------------%
\begin{proposition}\label{p2}
Suppose that $0 < |\alpha | < 4 $, and that $f, g$ and $h$ are
even functions from $X$ to $Y$ such that $f(0)=g(0)= h(0)=0$ and
\begin{eqnarray} \label{101}
N( f(x+y) + f(x-y) -2 g(x) - 2h(y), t) \geq N'(\varphi(x,y), t)
\end{eqnarray}
for each $x, y \in X$ and $t \in \mathbb{R}$. Then there is a
unique quadratic mapping $Q: X \to Y$ such that
\begin{eqnarray*}
N(Q(x) - f(x), t) \geq N^{''}(x, \frac{4- |\alpha |}{16}t),
\end{eqnarray*}
\begin{eqnarray*}
N(Q(x) - g(x), t) \geq N^{''}(x, \frac{12-3 |\alpha |}{52-
|\alpha |}t),
\end{eqnarray*}
\begin{eqnarray*}
N(Q(x) - h(x), t) \geq N^{''}(x, \frac{12-3 |\alpha |}{52-
|\alpha |}t),
\end{eqnarray*}
where $N^{''}(x, t)$ is defined by (\ref{106}).
\end{proposition}
%----------------------------------------------------------------------------------------------%
\begin{proof} Noting to (N3), it is sufficient to prove the theorem in the case that $0 < \alpha <$. Change the roles of $x$ and $y$ in (\ref{101}) to get
\begin{eqnarray}\label{102}
N( f(x+y) + f(x-y) -2 g(y) - 2h(x), t) \geq N'(\varphi(y, x), t).
\end{eqnarray}
Put $y=x$ in (\ref{101}) to obtain
\begin{eqnarray}
N(f(2x) - 2g(x) - 2 h(x), t )\geq N'(\varphi(x,x), t).
\end{eqnarray}
Put $x=0$ in (\ref{101}) to obtain
\begin{eqnarray}\label{103}
N( 2f(y) - 2 h(y), t) \geq N'(\varphi(0,y), t).
\end{eqnarray}
Similarly, putting $y=0$ in (\ref{101}) we get
\begin{eqnarray}\label{104}
N( 2 f(x) -2 g(x), t) \geq N'(\varphi(x,0), t).
\end{eqnarray}
Combining (\ref{102}), (\ref{103}), (\ref{104}) we get
\begin{eqnarray}\label{quadn}
N(f(x+y) - f(x-y) - 2f(x) - 2f(y), t ) &\geq& \min \{
N'(\varphi(x,y), t/3),\nonumber \\&& N'(\varphi(x,0), t/3),
N'(\varphi(0,y), t/3)\}.
\end{eqnarray}
Setting $y=x$ in (\ref{quadn}) we have
\begin{eqnarray}\label{105}
N(f( 2x ) - 4 f(x), t ) \geq N^{''}(x, t),
\end{eqnarray}
where $N^{''}(x, t)$ is defined by
\begin{eqnarray*}
N^{''}(x, t) =  \min \{ N'(\varphi(x,x), t/3), N'(\varphi(x,0),
t/3), N'(\varphi(0,x), t/3)\}.
\end{eqnarray*}
By (\ref{p2al}),
\begin{eqnarray}\label{107}
N^{''}(2^n x, t) = N^{''}(x, {t\over  \alpha^n}),
\end{eqnarray}
for each $n \geq 0$ and $x \in X$. It follows from (\ref{105}) and
(\ref{107}) that
\begin{eqnarray}\label{108}
N(f( 2^{n+1}x ) - 4 f(2^n x), t ) \geq N^{''}(x, {t\over
 \alpha^n}).
\end{eqnarray}
By (\ref{108}),
\begin{eqnarray*}
N({f( 2^{n+1}x )\over 4^{n+1}} - {f(2^{n} x)\over 4^n}, t ) &=&
N(f( 2^{n+1}x ) - 4 f(2^n x), 4^{n+1} t) \\ &\geq& N^{''}(x,
{{4^{n+1}t\over  \alpha^n}}).
\end{eqnarray*}
or equivalently, \begin{eqnarray*} N({f( 2^{n+1}x )\over 4^{n+1}}
- {f(2^{n} x)\over 4^n}, { \alpha^n t\over 4^{n+1}} ) \geq
N^{''}(x, t).
\end{eqnarray*}
Therefore for each $n > m \geq 0$,
\begin{eqnarray}\label{110}
N({f( 2^n x )\over 4^n } - {f(2^m x)\over 4^m}, \sum_{k=m+1}^n {
\alpha^{k - 1} t\over 4^{k}} ) &=&   N( \sum_{k=m+1}^n ({f( 2^k x
)\over 4^k } - {f(2^{k-1} x)\over 4^{k-1}}), \sum_{k=m+1}^n {
\alpha^{k - 1} t\over 4^{k}} ) \nonumber \\ &\geq & \min
\bigcup_{k=m+1}^n \{N {f( 2^{k}x )\over 4^{k}} - {f(2^{k-1}
x)\over 4^{k-1}}, { \alpha^{k - 1} t\over 4^{k}} )\} \\ &\geq&
N^{''}(x, t).\nonumber
\end{eqnarray}
Let $t_0 > 0$, $\varepsilon >0$ be given. Since $\lim_{t \to
\infty}N^{''}(x, t) = 1$ there is some $t_1 > t_0$ such that
$N^{''}(x, t_1 ) > 1 - \varepsilon$. The convergence of the series
$\sum_{k=1}^\infty { \alpha^{k - 1} \over 4^{k}}t_1$ gives some
$n_0$ such that $\sum_{k=m+1}^n { \alpha^{k - 1} \over 4^{k}}t_1 <
t_0$ for each $n > m \geq n_0$. It follows that for each $n > m >
n_0$,
\begin{eqnarray*}
N({f( 2^n x )\over 4^n } - {f(2^m x)\over 4^m}, t_0) &\geq&  N({f(
2^n x )\over 4^n } - {f(2^m x)\over 4^m}, \sum_{k=m+1}^n
{\alpha^{k - 1} \over 4^{k}}t_1)
\\ &\geq& N^{''}(x, t_0) \\ &>& 1 - \varepsilon.
\end{eqnarray*}
This shows that $\{ {f( 2^n x )\over 4^n }\}$ is Cauchy sequence
in the fuzzy Banach space $(Y, N)$, therefore it is convergence
to some $Q(x)$. So we can define a mapping $Q: X \to Y$ by
$Q(x)\colon= N-\lim_{n\to\infty}\frac{f(2^nx)}{4^n}$.

Fix $x, y \in X$ and $t>0$. It follows from (\ref{quadn}) that
\begin{eqnarray}\label{113}
N({f(2^n(x+y))\over 4^n} + {f(2^n(x-y))\over 4^n}&-& 2{f(2^n
x)\over
4^n} - 2{f(2^n y)\over 4^n}, \frac{t}{5})\\&=& N ( f(2^n(x+y)) + f(2^n(x-y))\nonumber\\
&&- 2f(2^n x) - 2f(2^n y), \frac{4^n t}{5})\nonumber \\
&\geq& \min \{ N'( \varphi(x, y ), {4^n t\over 15
\alpha^n}),\nonumber
\\&& N'( \varphi(x, 0 ), {4^n t\over 15  \alpha^n}),  N'(
\varphi(0, y ), {4^n t\over 15  \alpha^n})\}\nonumber
\end{eqnarray}
for all $n$. Moreover,
\begin{eqnarray}\label{114}
N(Q(x+y) &+& Q(x-y)- 2Q(x) - 2Q(y), t ) \geq \min\{ N(Q(x+y) - {
f(2^n(x+y)) \over 4^n}, {t\over 5} ),\nonumber \\&&  N(Q(x-y) - {
f(2^n(x-y)) \over 4^n}, {t\over 5} ), N(2Q(x) - 2{f(2^n x)\over
4^n}, {t\over 5}), \nonumber\\&& N(2Q(y) - 2{f(2^n y)\over 4^n},
{t\over 5}),
\\&&N({f(2^n(x+y))\over 4^n} + {f(2^n(x-y))\over 4^n} - 2{f(2^n x)\over
4^n} - 2{f(2^n y)\over 4^n}, {t\over 5})\} \nonumber
\end{eqnarray}
for all $n$. Since each factor in the right hand side of
(\ref{113}) and (\ref{114}) tends to $1$ as $n \to \infty$, one
can easily see that $N(Q(x+y) + Q(x-y) - 2Q(x) - 2Q(y), t)=1$
whence
$$Q(x+y) + Q(x-y) = 2Q(x) + 2Q(y).$$
Furthermore, using (\ref{110}) with $m = 0$, we see that for large
$n$,
\begin{eqnarray}\label{Qft}
N(Q(x) - f(x), t) &\geq& \min \{ N(Q(x) - {f(2^n x)\over 4^n},
{t\over 2}), N( {f(2^n x)\over 4^n} - f(x), {t \over 2}) \}\nonumber\\
&\geq& \min \{ N(Q(x) - {f(2^n x)\over 4^n}, {t\over 2}),
N^{''}(x, {4t \over  \sum^n_{k=1}({ \alpha \over 4})^{k-1}})\}\nonumber\\
&\geq& N^{''}(x, {4t \over \sum^\infty_{k=0}( { \alpha \over
4})^{k}})\\
&=& N^{''}(x, \frac{4- \alpha}{16}t).
\end{eqnarray}
It follows from (\ref{104}) and (\ref{Qft}) that
\begin{eqnarray*}
N(Q(x) - g(x), \frac{52- \alpha}{48}t) &\geq&
\min\{N(Q(x)-f(x), t), N(f(x)-g(x),\frac{4- \alpha}{48}t)\}\\
&\geq& \min\{ N^{''}(x, \frac{4- \alpha}{16}t), N'(\varphi(x,0),
\frac{4- \alpha}{48}t)\}\\&\geq& N^{''}(x, \frac{4- \alpha}{16}t)
\end{eqnarray*}
whence
\begin{eqnarray*}
N(Q(x) - g(x), t) \geq N^{''}(x, \frac{12-3 \alpha}{52- \alpha}t).
\end{eqnarray*}
A similar inequality holds for $h$. The uniqueness assertion can
be proved by a known strategy as in Theorem \ref{t1}.
\end{proof}

%----------------------------------------------------------------------------------------------------------
\begin{theorem}
Let $ |\alpha | < 2 $, let $f$ be a mapping from $X$ to $Y$ such
that $f(0)=0$ and
\begin{eqnarray} \label{201}
N( f(x+y) + f(x-y) -2 f(x) - 2f(y), t) \geq N'(\varphi(x,y), t)
\end{eqnarray}
for all $x, y \in X$ and all $t >0$. Then there are unique
mappings $T$ and $Q$ from $X$ to $Y$ such that $T$ is additive,
$Q$ is quadratic and
\begin{eqnarray*}
N(f(x)-T(x)-Q(x), t) \geq M(x, \min\{ \frac{2-\alpha}{8},
\frac{4-\alpha}{32}\}\, t)\quad (x \in X, t>0),
\end{eqnarray*}
where
\begin{eqnarray*} M(x, t) &=&  \min \{ N'(\varphi(x,x),{t/ 3}),  N'(\varphi(-x, -x), {t/ 3})
, \\&& N'(\varphi(x,0), t/3), N'(\varphi(0,x), t/3),\\&&
N'(\varphi(-x,0), t/3), N'(\varphi(0,-x), t/3)\}.
\end{eqnarray*}
\end{theorem}
%----------------------------------------------------------------------------------
\begin{proof}
Passing to the odd part $f^0$ and even part $f^e$ of $f$ we
deduce from (\ref{201}) that
\begin{eqnarray*}
N( f^o(x+y) + f^o(x-y)&-& 2 f^o(x) - 2f^o(y), t)\\ &\geq&
\min\{N'(\varphi(x,y), t), N'(\varphi(-x,-y), t)\}
\end{eqnarray*}
and
\begin{eqnarray*}
N( f^e(x+y) + f^e(x-y)&-& 2 f^e(x) - 2f^e(y), t)\\ &\geq&
\min\{N'(\varphi(x,y), t), N'(\varphi(-x,-y), t)\}.
\end{eqnarray*}
Using the proofs of Propositions \ref{p1} and \ref{p2} we get
unique additive mapping $T$ and unique quadratic mapping $Q$
satisfying
\begin{eqnarray*}
N(f^o(x)-T(x), t) \geq M(x, \frac{2- |\alpha |}{4}t),
\end{eqnarray*}
and
\begin{eqnarray*}
N(f^e(x)-Q(x), t) \geq M(x, \frac{4- |\alpha |}{16}t).
\end{eqnarray*}
Therefore
\begin{eqnarray*}
N(f(x)-T(x)-Q(x), t) &\geq& \min\{N(f^o-T(x), \frac{t}{2}),
N(f^e-Q(x), \frac{t}{2})\}\\ &\geq& \min\{M(x, \frac{2- |\alpha
|}{8}t), M(x, \frac{4- |\alpha |}{32}t)\}\\ &=&M(x, \min\{
\frac{2-\alpha}{8}, \frac{4-\alpha}{32}\}\, t)\quad (x \in X,
t>0).
\end{eqnarray*}
\end{proof}
%----------------------------------------------------------------------------------------------------------

The following example provides an illustration.

\begin{example} Let $(X, <.,.>)$ be an inner product space, $Y$ be a normed
space and $Z$ be the real line $\mathbb{R}$. Let $N$ and
$N^{\prime}$ be the fuzzy norms  on $Y$ and $\mathbb{R}$, defined
by Example \ref{exam} with $k=1$, respectively. Suppose that the
fuzzy metric $N$, makes $Y$ into a fuzzy Banach space. Fix
elements $x_0$, $y_0$ and $z_0$ in $Y$ and $a$ in $X$ and define
$$f(x)= < x, a > x_0 + \|x\|^2 y_0 + \sqrt{\|x\|}\,z_0,$$
$$g(x) =  < x, a > x_0 + \|x\|^2 y_0,$$
$$h(x) = \|x\|^2 y_0 + \sqrt{\| x\|}\,z_0,$$
$$\varphi(x, y) =( \sqrt{\|x+y\|} + \sqrt{ \|x-y\|} - 2
 \sqrt{\|y\|})\, \| z_0\|$$
for each $x,y \in X$. One can easily verified that
$$f(x+y) + f( x-y) - 2g(x) - 2h(y) =( \sqrt{ \|x + y\|} +
 \sqrt{\|x - y \|} - 2 \sqrt{\|y\|})\,z_0$$
for each $x, y \in X$. Therefore
$$N( f(x+y) + f( x-y) - 2g(x) - 2h(y), t ) = N^\prime (\varphi(x,
 y), t)$$ for each $x, y \in X$ and $ t \in \mathbb{R}$.  Moreover,
$\varphi(2x, 2y) = \sqrt{2}\varphi(x, y)$ for each $x, y \in X$.

Therefore the conditions of Propositions \ref{p1} and \ref{p2} for
$f$, $g$, $h$  and $ |\alpha | = \sqrt{2} < 2$ are satisfied. It
follows that odd  and even parts of $f$ can be approximated by
linear and quadratic functions, respectively. In fact  $f^o$, the
odd part of $f$, is equal to $f^o(x)= <x, a>
 x_0$ is linear and the even part of $f$, $f^e$, is equal to
$f^e(x)= \|x\|^2 y_0 + \sqrt{\|x\|} z_0$ contains a quadratic
$Q(x) = \|x\|^2 y_0$ and
$$N(f^e(x) - Q(x), t ) =  N^\prime(\sqrt{\|x\|}\,\|z_0\|, t) \geq N^{''}(x,\frac{4-\sqrt{2}}{16}t).$$
\end{example}

\end{document}